\documentclass[11pt]{article}
\pdfoutput=1

\usepackage{geometry}
\usepackage{amsmath}
\usepackage{amssymb}
\usepackage{amsthm}
\usepackage{mathtools}
\usepackage{mathrsfs}  

\usepackage{graphicx}
\usepackage{tikz}
\usepackage{tikz-cd}   
\usetikzlibrary{arrows.meta, positioning, decorations.markings, calc}
\usepackage[x11names,svgnames]{xcolor}

\usepackage{microtype}
\usepackage[T1]{fontenc}
\usepackage{lmodern}

\usepackage[numbers,sort&compress]{natbib}

\usepackage{hyperref}
\hypersetup{colorlinks=true,
            linkcolor=blue!70!black,
            citecolor=blue!70!black,
            urlcolor=blue!70!black}
\makeatletter
\g@addto@macro{\UrlBreaks}{\UrlOrds}
\makeatother

\usepackage{doi}
\usepackage{aliascnt}
\usepackage[capitalise, nameinlink]{cleveref}

\usepackage{enumitem}
\usepackage{booktabs}  
\usepackage{appendix}  
\usepackage{caption}

\usepackage{amsfonts}
\usepackage{mathdots}

\usepackage{subcaption}   
\usepackage{float}         
\usepackage{wrapfig}       
\usepackage{adjustbox}     
\usepackage{stackengine}   
\usepackage{scalerel}      

\usepackage{comment}
\usepackage[colorinlistoftodos]{todonotes}

\usepackage{nicematrix}

\newtheoremstyle{ghriststyle}
    {12pt}                    
    {12pt}                    
    {\itshape}                
    {}                        
    {\scshape}                
    {.}                       
    {.5em}                    
    {}                        

\newtheoremstyle{ghristdefstyle}
    {12pt}                    
    {12pt}                    
    {\normalfont}             
    {}                        
    {\scshape}                
    {.}                       
    {.5em}                    
    {}                        

\newtheoremstyle{ghristremarkstyle}
    {8pt}                     
    {8pt}                     
    {\normalfont}             
    {}                        
    {\itshape}                
    {.}                       
    {.5em}                    
    {}                        

\theoremstyle{ghriststyle}
\newtheorem{theorem}{Theorem}[section]
\newaliascnt{lemma}{theorem}
\newtheorem{lemma}[lemma]{Lemma}
\aliascntresetthe{lemma}
\crefname{lemma}{Lemma}{lemmas}
\newaliascnt{proposition}{theorem}
\newtheorem{proposition}[proposition]{Proposition}
\aliascntresetthe{proposition}
\crefname{proposition}{Proposition}{propositions}
\newaliascnt{corollary}{theorem}
\newtheorem{corollary}[corollary]{Corollary}
\aliascntresetthe{corollary}
\crefname{corollary}{Corollary}{corollaries}
\newtheorem{conjecture}[theorem]{Conjecture}

\theoremstyle{ghristdefstyle}
\newtheorem{definition}[theorem]{Definition}
\newtheorem{example}[theorem]{Example}

\theoremstyle{ghristremarkstyle}
\newaliascnt{remark}{theorem}
\newtheorem{remark}[remark]{Remark}
\aliascntresetthe{remark}
\crefname{remark}{Remark}{remarks}
\newaliascnt{note}{theorem}

\aliascntresetthe{note}
\crefname{note}{Note}{notes}

\stackMath

\definecolor{oreomint}{HTML}{4CAF50}  
\definecolor{oreogold}{HTML}{D4A843}  

\definecolor{bc_blue}{HTML}{2196F3}
\definecolor{bc_orange}{HTML}{F58B28}
\definecolor{bc_green}{HTML}{4CAF50}
\definecolor{bc_magenta}{HTML}{E91E63}
\definecolor{bc_purple}{HTML}{9C27B0}
\definecolor{bc_red}{HTML}{E8352E}
\definecolor{bc_yellow}{HTML}{F5D623}

\newcommand{\birthdaycake}{%
  \textbf{%
    \textcolor{bc_blue}{B}%
    \textcolor{bc_orange}{i}%
    \textcolor{bc_green}{r}%
    \textcolor{bc_magenta}{t}%
    \textcolor{bc_purple}{h}%
    \textcolor{bc_orange}{d}%
    \textcolor{bc_red}{a}%
    \textcolor{bc_yellow}{y }%
    \textcolor{bc_green}{C}%
    \textcolor{bc_orange}{a}%
    \textcolor{bc_magenta}{k}%
    \textcolor{bc_purple}{e}%
  } Oreo}

\newcommand{\framedimg}[2][]{%
  \adjustbox{frame=0.4pt, cfbox=black!40 0.4pt 2pt}{%
    \includegraphics[#1]{#2}%
  }%
}

\makeatletter

\newcommand{\thanksblock}[1]{\gdef\@thanksblock{#1}}
\gdef\@thanksblock{} 

\renewcommand{\maketitle}{%
  \begin{center}%
    {\LARGE\bfseries \@title \par}%
    \vskip 1em%
    {\large \lineskip .5em%
      \begin{tabular}[t]{c}%
        \@author
      \end{tabular}\par}%
    \ifx\@thanksblock\@empty\else
      \vskip .75em%
      {\small \@thanksblock\par}%
    \fi
  \end{center}%
  \vskip 1.5em%
}

\makeatother

\date{} 

\usepackage{titlesec}

\titleformat{\section}
    {\normalfont\large\bfseries}
    {\thesection.}
    {0.5em}
    {}

\titleformat{\subsection}
    {\normalfont\normalsize\bfseries}
    {\thesubsection.}
    {0.5em}
    {}

\titleformat{\subsubsection}
    {\normalfont\normalsize\itshape}
    {\thesubsubsection.}
    {0.5em}
    {}

\renewenvironment{abstract}{%
    \begin{center}
        \textsc{Abstract}
    \end{center}
    \begin{quotation}
    \noindent
    \small
}{%
    \end{quotation}
    \vspace{1em}
    \normalsize
}

\title{The $\infty$-Oreo\textsuperscript{\textregistered}}

\author{Vicente Bosca}

\thanksblock{%
Department of Mathematics, University of Pennsylvania\\
\texttt{vicenteg@sas.upenn.edu}
}

\begin{document}
\maketitle

\renewcommand{\thefootnote}{}
\footnotetext{\scriptsize
Oreo\textsuperscript{\textregistered},
Double Stuf\textsuperscript{\textregistered},
Mega Stuf\textsuperscript{\textregistered},
Oreo Loaded\textsuperscript{\textregistered}, Mint Oreo\textsuperscript{\textregistered}, Birthday Cake Oreo\textsuperscript{\textregistered}, and Golden Oreo\textsuperscript{\textregistered} are trademarks of
Mondel\={e}z International, Inc.\
M\&M's\textsuperscript{\textregistered} and
Crunchy Cookie M\&M's\textsuperscript{\textregistered} are trademarks of
Mars, Inc. \ Ben \& Jerry's\textsuperscript{\textregistered} is a registered trademark of Unilever. 
This work is not affiliated with, endorsed by, or sponsored by
any of these companies.}
\renewcommand{\thefootnote}{\arabic{footnote}}


\begin{abstract}
What happens when a food product contains a version of itself?  The
Oreo Loaded---a cookie whose filling contains real Oreo cookie
crumbs---can be viewed as the result of mixing a Mega Stuf Oreo into a Mega Stuf
Oreo.  Iterating this process yields a sequence of increasingly
self-referential cookies; taking the limit gives the $\infty$-Oreo.
We model the iteration as an affine recurrence on the creme fraction
of the filling, prove convergence, and compute the limit exactly: the
stuf of the $\infty$-Oreo is approximately $95.8\%$~creme and
$4.2\%$~wafer.  We then extend the framework to pairs of foods that
reference each other, deriving a coupled recursion whose fixed point
defines a \emph{bi-$\infty$ food}, and illustrate the construction
with M\&M Cookies and Crunchy Cookie M\&M's.  Finally, we classify
$\infty$-foods by the number of foods in the recursion and introduce
\emph{homological foods}, whose recursive structure is governed by
cycles in a directed graph of commercially available products.  We
close with a conjecture.  All products used in this paper can be
purchased at a supermarket.
\end{abstract}



\section{Introduction}\label{sec:introduction}

The Oreo cookie, manufactured by Mondelēz International since 1912, is
among the most commercially successful food products in history.  It is
also the subject of a small but dedicated body of scientific
work~\cite{anderson2013, teenvogue, james_lamers, mit_oreo}.  We
contribute to this literature by addressing the following question: can
Oreo produce the best possible Oreo-flavored cookie?
We show that the answer is yes, and that it can be computed exactly.  By
modeling the iterative process of making an Oreo-flavored Oreo as a
recurrence relation and taking the limit, we obtain what we call the
\textbf{$\infty$-Oreo}: the cookie whose flavor is, in a precise sense,
as Oreo as a cookie can be.

Specific contributions of this work include the following.

\begin{enumerate}[label=\arabic*.]
    \item In \S\ref{sec:preliminaries}, we introduce the anatomy of the
    Oreo cookie and review previous empirical work on its composition.

    \item In \S\ref{sec:infinity-oreo}, we set up the Oreo recursion,
    solve it, and compute the composition of the $\infty$-Oreo.

    \item In \S\ref{sec:bi-infinity}, we extend the framework to systems
    of two foods that reference each other, which we call \emph{bi-$\infty$
    foods}, and illustrate the construction with M\&M and Cookies.

    \item In \S\ref{sec:classification}, we classify $\infty$-foods by the
    number of foods involved in the recursion and postulate the existence of
    \emph{homological foods}: $\infty$-foods whose recursive structure is
    governed by the topology of an underlying directed graph.

    \item In \S\ref{sec:discussion}, we discuss open problems and future
    directions.
\end{enumerate}

\subsection{Motivations}\label{sec:motivations}

Readers who need no motivation for recursive snacking may skip to
\S\ref{sec:preliminaries}.

\medskip

A successful food company must periodically introduce new flavors.  In the
early stages this is straightforward: one selects a well-known flavor and
incorporates it into the product.  Over time, however, the space of
obvious candidates is exhausted, and the process of flavor innovation
becomes increasingly fun and abstract.

This progression can be understood through the \emph{Yankee Candle Levels
of Abstraction}, a classification framework introduced on
Twitter~\cite{yankeecandle} for the increasingly detached relationship
between a scented candle and the real-world object it claims to evoke.
The same hierarchy applies to cookie flavors.\footnote{The Yankee Candle
framework is closely related to Baudrillard's theory of
simulacra~\cite{baudrillard}, known as a meme~\cite{reddit_baudrillard} through the example that begins with a
fish and ends with Swedish Fish--flavored Oreos.}

\begin{wrapfigure}{R}{0.5\textwidth}
  \centering
  \vspace{-10pt}
  \framedimg[width=\linewidth]{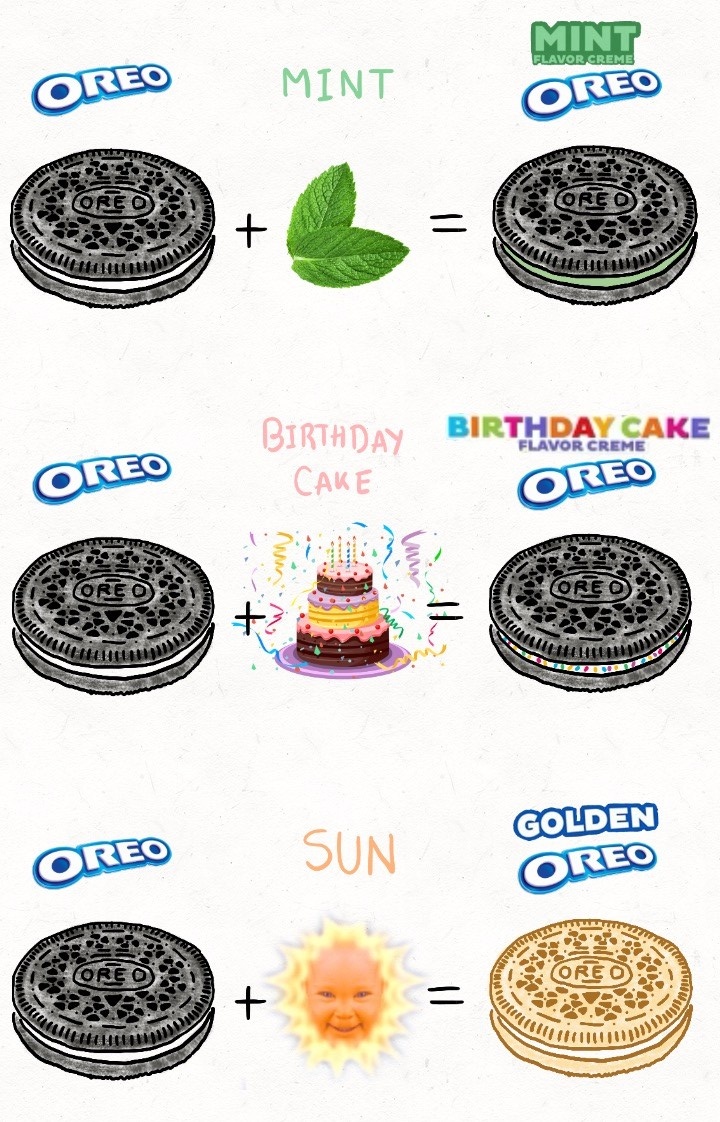}
  \caption{Flavor innovation at three levels of abstraction.}\label{fig:oreo-flavors}
  \vspace{-8pt}
\end{wrapfigure}

At \textbf{Level~1}, the scent represents a physical object whose aroma is
one of its distinctive properties: a candle that smells like black
cherries.  This is the straightforward phase.

An Oreo equivalent is the
\textcolor{oreomint}{\textbf{Mint Creme Oreo}}: take a flavor from the real world and mix it
into the filling.

At \textbf{Level~2}, the scent represents a physical object in a specific
state, where the aroma evokes the state rather than the object itself: in a
``clean cotton'' candle, it is the
\emph{cleanness} of the cotton that constitutes the distinctive aroma, not the cotton itself.
An Oreo equivalent is the \birthdaycake.  Note that the
flavor of the cake is irrelevant; what matters is that we are at a
\emph{birthday}.

At \textbf{Level~7}, the scent represents an intrinsic property that itself
has no aroma.  Yankee Candle achieved this with a product labeled ``All Is
Bright.''  An Oreo equivalent is the 
\textcolor{oreogold}{\textbf{Golden Oreo}}, a product best
described as what happens when we mix an Oreo with the Sun.

At \textbf{Level~$\boldsymbol{\Omega}$}, the scent represents total
abstraction: the absence of any referent.  Yankee Candle's entry at this
level is a candle called ``Sweet Nothings.''  This is checkmate for Oreo.  They cannot sell a nothing-flavored cookie; it would just be a cracker.

Rather than continuing to climb the ladder
of abstraction, Oreo's resolution is to abandon the hierarchy entirely and pose a different
question: what is the best possible flavor for an Oreo?  From Oreo's
perspective, the answer is clear---Oreo itself.  The goal, then, is to
produce a cookie that tastes more like Oreo than an ordinary Oreo does.
What this means is not at all obvious.  To make it precise, we need some
preliminaries.


\section{Preliminaries---The Anatomy of an Oreo}\label{sec:preliminaries}

In this section we establish notation and review the empirical data that
will serve as inputs to the recursion in \S\ref{sec:infinity-oreo}.  

\subsection{Components and notation}\label{sec:components}

The Oreo is a sandwich cookie consisting of two components.

\begin{definition}[Wafer and Stuf]\label{def:wafer-stuf}
The \textbf{wafer} is the dark biscuit component of an Oreo cookie.  The
\textbf{stuf} is the white filling.\footnote{The official Mondelēz
spelling is ``Stuf,'' not ``stuff.''  The filling is marketed as
\emph{creme} (not \emph{cream}), since it contains no dairy, making the
Oreo---perhaps surprisingly---a vegan cookie.}
\end{definition}

Every Oreo can be described by how much of its mass comes from each component.

\begin{definition}[Stuf fraction]\label{def:stuf-fraction}
Let $s$ denote the mass fraction of an Oreo cookie attributable to the
stuf. $1 - s$ corresponds to the part attributable to the wafer.
\end{definition}

The composition of any Oreo cookie is fully determined by a single
number $s \in [0,1]$.  At one extreme, $s = 0$ is a cookie with no
filling---just two wafers pressed together, a product that would
delight the small but passionate community of those who maintain that the
wafer is the only part worth eating.  At the other, $s = 1$ is a
free-standing disk of stuf with no wafer at all, the platonic ideal for
stuf purists.

\begin{figure}[ht]
  \centering
  \framedimg[width=0.85\textwidth]{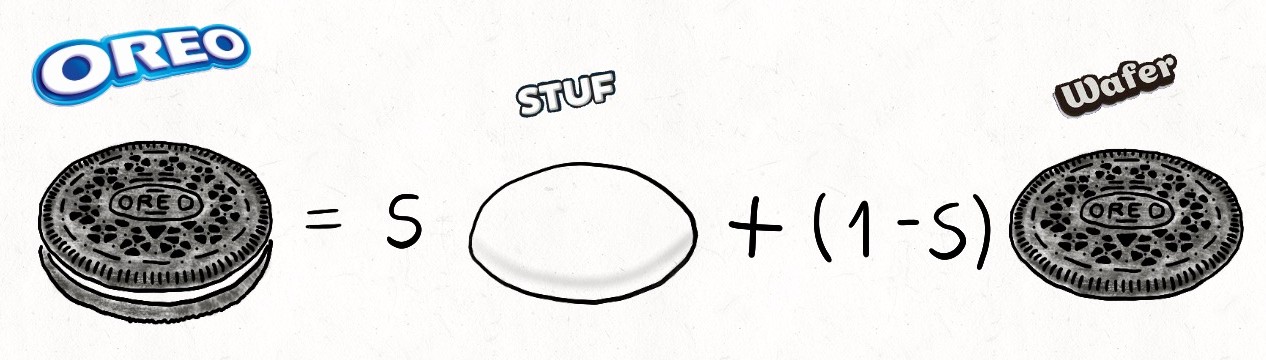}
  \caption{Original Oreo broken down into its stuf fraction~$s$ and wafer fraction~$1-s$.}\label{fig:oreo-decomposition}
\end{figure}

The composition of the original Oreo has been measured repeatedly.
According to data published by \emph{Teen
Vogue}~\cite{teenvogue}---in an article titled ``10 Things You Definitely Didn't
Know About Oreos''---the original
Oreo has
\begin{equation}\label{eq:original-oreo}
    s = 0.29.
\end{equation}
That is, the original Oreo is roughly 29\% stuf and 71\% wafer by mass.
We take these values as given throughout the paper.

\subsection{The Oreo product line}\label{sec:product-line}

To produce an Oreo-flavored Oreo, one must enhance the cookie with more of one of its two components: add more stuf,
or add more wafer.  Oreo started with the stuf---the preferred
component among the larger of the two Oreo factions.

\subsubsection*{More stuf.}

\begin{wrapfigure}{R}{0.5\textwidth}
  \centering
  \vspace{-10pt}
  \framedimg[width=\linewidth]{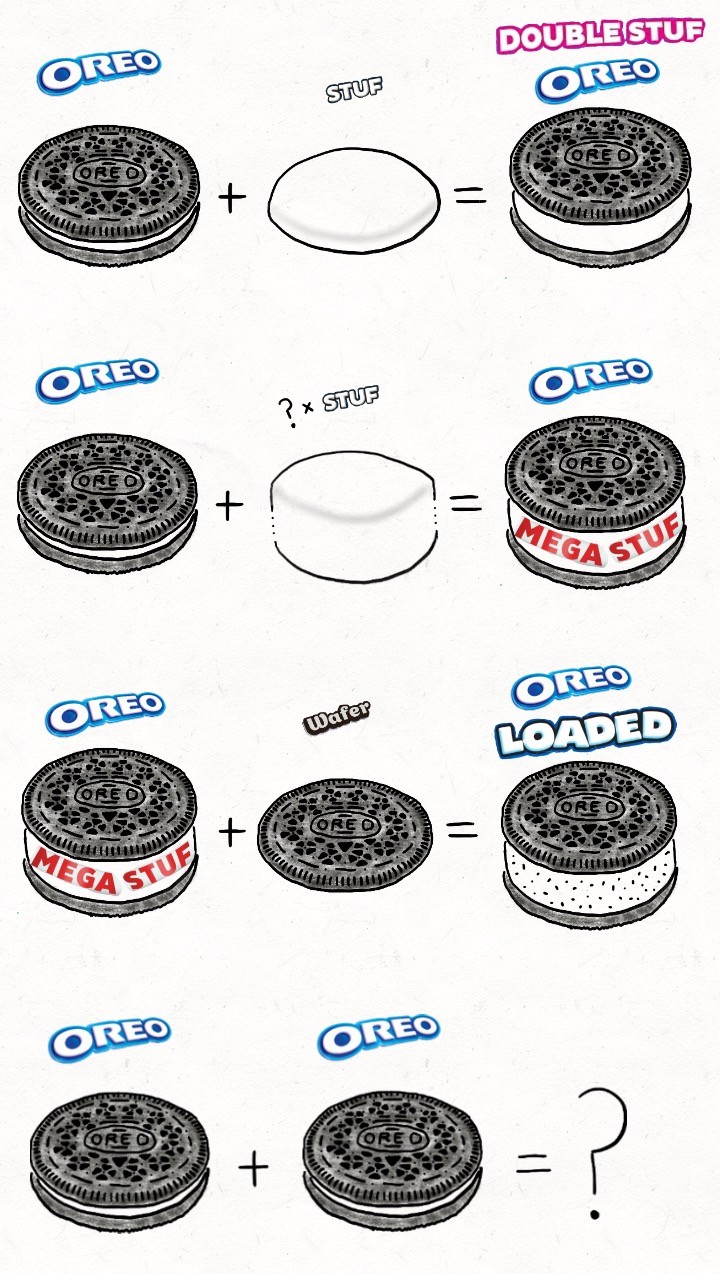}
  \caption{The Oreo product line as a sequence of enhancements.}\label{fig:oreo-oreo}
  \vspace{-6pt}
\end{wrapfigure}

The simplest approach is to increase the amount of filling.  The
\textbf{Double Stuf Oreo}, introduced in 1974, does exactly this: the
wafers are unchanged, but the stuf is ostensibly doubled.

Whether the stuf is \emph{actually} doubled was the subject of a 2013 experiment by Dan Anderson, a high school mathematics teacher in Queensbury, New York, and his students~\cite{anderson2013}. The result: a Double Stuf Oreo contained only 1.86 times the stuf of an original.

The experiment attracted the attention of ABC News~\cite{abcnews}, whose coverage opened with the words: ``If you ever thought high school math wasn't useful in real life\dots''  Mondelēz responded by maintaining that the recipe does call
for double the filling~\cite{abcnews}, and a subsequent study at Clemson
University with a larger sample supported this
claim~\cite{clemson2014}.  

Anderson also measured the \textbf{Mega Stuf Oreo}, a product released
around the same time with a visibly larger but unspecified amount of
filling.  The Mega Stuf turned out to contain $2.68$ times the stuf of
an original~\cite{anderson2013}---establishing, at last, a concrete
meaning for the prefix \emph{mega}.

\subsubsection*{More wafer.}

One could instead enhance the cookie with additional wafer, for instance
by mixing cookie crumbs into the filling.  With the original Oreo as a
base, however, this modification is nearly imperceptible.  The original is
already 71\% wafer by mass; adding a small amount of crumbs to an already
wafer-dominated cookie produces a flavor difference that is, at best,
negligible.

The situation changes if one starts from a stuf-enhanced base.  A Mega
Stuf Oreo has a substantially larger filling, enough that mixing in
cookie crumbs produces a recognizable difference.  Oreo exploited
precisely this idea with the \textbf{Oreo Loaded}: a cookie with Mega
Stuf--level creme in which real Oreo cookie crumbs have been mixed into
the filling.

Note, however, that every product described so far enhances only
\emph{one} component of the Oreo.  The Double Stuf and Mega Stuf add
more creme; the Oreo Loaded adds more wafer.
None of them puts an entire Oreo inside another Oreo. 

This is where a mathematician's perspective becomes useful. The Oreo Loaded's filling can be viewed as creme from a base Mega Stuf Oreo mixed with a whole second Mega Stuf Oreo---crumbs, creme, and all. The Oreo Loaded is, in this sense, a Mega Stuf Oreo mixed with a Mega Stuf Oreo: the first Oreo-flavored Oreo.

But this is only the \emph{first}.  The Oreo Loaded contains another Oreo inside
it, but what if it contained another Oreo Loaded?  A truly Oreo-flavored Oreo would contain an Oreo-flavored
Oreo, which would contain an Oreo-flavored Oreo, and so on.  To reach
this ideal, we must iterate the process.  This is the subject of the next
section.


\section{The $\infty$-Oreo}\label{sec:infinity-oreo}

\begin{wrapfigure}{R}{0.5\textwidth}
  \centering
  \vspace{-10pt}
  \framedimg[width=\linewidth]{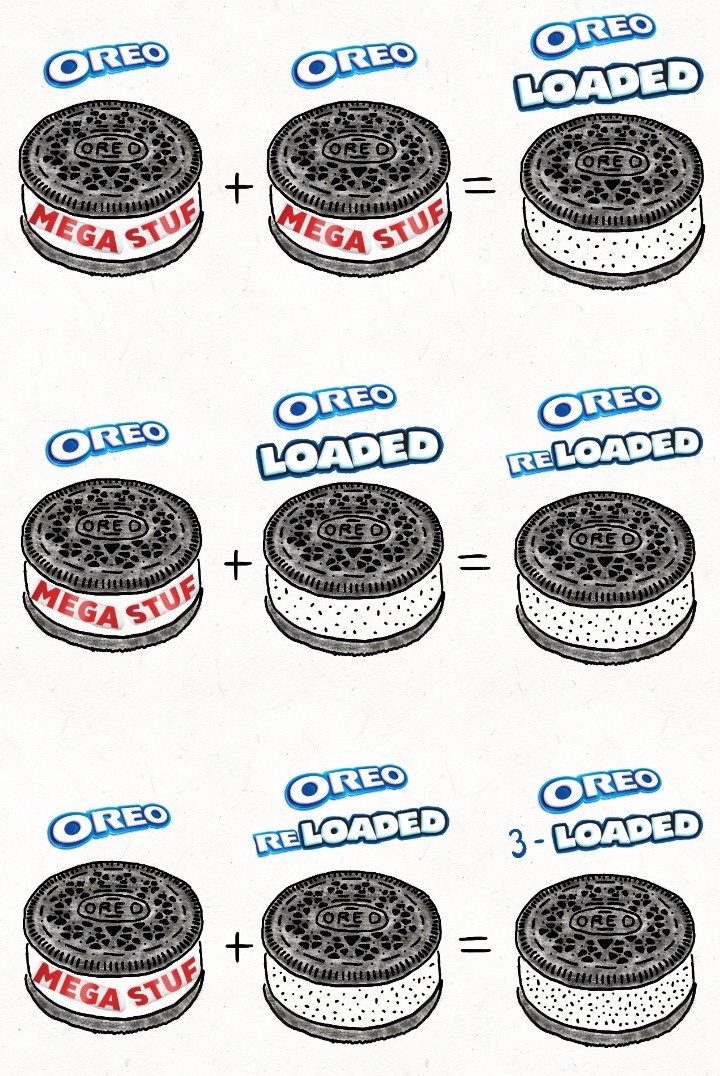}
  \caption{The first three iterations.}\label{fig:oreo-iteration}
  \vspace{-8pt}
\end{wrapfigure}

We saw in \S\ref{sec:product-line} that the Oreo Loaded can be regarded
as the result of mixing a Mega Stuf Oreo into a Mega Stuf Oreo.  In this
section, we iterate the process and compute the composition of the limit.

\subsection{The recursion}\label{sec:recursion}

The Oreo Loaded is the first Oreo-flavored Oreo, but its interior Oreo
is just an ordinary Mega Stuf.  A more thoroughly Oreo-flavored Oreo
would contain, inside its filling, not just any Oreo, but an
Oreo-flavored Oreo.  We can arrange this by repeating the construction.

\subsubsection*{The Oreo Reloaded.}

Take a Mega Stuf Oreo as the base cookie, but this time, instead of
mixing in a plain Mega Stuf, mix in an Oreo Loaded.  The result is a
cookie whose filling contains creme and crumbs from a cookie that itself
contains creme and Oreo crumbs.  We call this the \textbf{Oreo Reloaded}:
an Oreo-flavored Oreo-flavored Oreo.

\subsubsection*{The Oreo 3-Loaded.}

There is no reason to stop.  Take another Mega Stuf base and mix in the
Oreo Reloaded.  The result is the \textbf{Oreo 3-Loaded}: a cookie whose
filling is made from creme and crumbs from a Reloaded, which was itself
made from crumbs of a Loaded, which was made from crumbs of a Mega Stuf.

\subsubsection*{The general case.}

The pattern is now clear.

\begin{definition}[The Oreo $n$-Loaded]\label{def:n-loaded}
For $n \geq 1$, the \textbf{Oreo $n$-Loaded} is the cookie obtained by
taking a Mega Stuf Oreo as the base and mixing an Oreo $(n{-}1)$-Loaded
into its filling.
\end{definition}
 
Note that this construction naturally identifies the Mega Stuf Oreo with
the Oreo $0$-Loaded, the Oreo Loaded with the Oreo $1$-Loaded, and the
Oreo Reloaded with the Oreo $2$-Loaded.

As $n$ increases, the
cookie mixed into the filling is itself the product of more and more
iterations, so the filling becomes an increasingly deep nesting of Oreo
within Oreo within Oreo.

This is a recurrence relation, and recurrence relations have limits.

\begin{definition}[The $\infty$-Oreo]\label{def:infinity-oreo}
The \textbf{$\infty$-Oreo} is the limit of the sequence of Oreo
$n$-Loaded cookies as $n \to \infty$:
\[
    \text{$\infty$-Oreo} \;\coloneqq\; \lim_{n \to \infty}\, \text{Oreo $n$-Loaded}.
\]
\end{definition}

The $\infty$-Oreo is, by construction, a cookie whose filling contains an
$\infty$-Oreo.  It is Oreo-flavored all the way down.

\begin{figure}[ht]
  \centering
  \framedimg[width=0.75\textwidth]{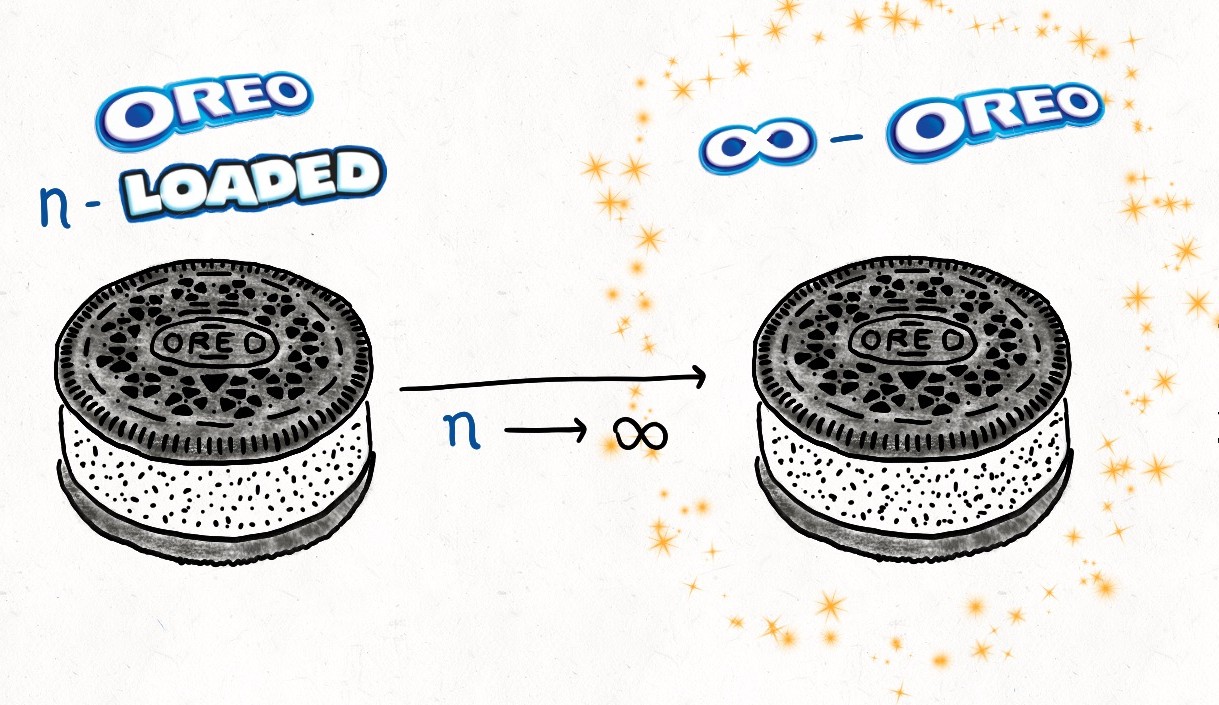}
  \caption{The $\infty$-Oreo: the limit of the Oreo $n$-Loaded as
  $n \to \infty$.}\label{fig:infinity-oreo}
\end{figure}

\medskip

To make this precise, we need to define what ``mixing'' means in terms of
the composition of the resulting cookie.

\subsection{The composition equations}\label{sec:composition-equations}

To compute the composition of the $\infty$-Oreo, we need to track how
the filling changes across iterations.  Every cookie in the sequence has
a stuf made of some mixture of creme and wafer crumbs; as $n$~increases,
the proportions shift.  Our goal is to determine the \emph{creme fraction}
of the stuf at each step, and take the limit.

This requires two measured quantities and one equation.  We introduce the
measurements first.

\subsubsection*{The Mega Stuf composition.}

Write $S_n$ for the filling inside the Oreo $n$-Loaded and $O_n$ for the
whole cookie.  Lamers~\cite{james_lamers} weighed the separated
components of five Oreo varieties on a kitchen scale.  Write $m_s$
for the mass of the Mega Stuf filling and $m_w$ for the combined mass
of the two wafers:
\begin{equation}\label{eq:mS-mW}
    m_s = 10\;\text{g}, \qquad m_w = 8\;\text{g}.
\end{equation}

\begin{definition}[Stuf fraction]\label{def:mega-stuf-fraction}
For each $n \geq 0$, let $m_n$ denote the mass fraction of the Oreo
$n$-Loaded attributable to its stuf.  The remaining fraction $1 - m_n$
is attributable to the wafer.
\end{definition}

For the Mega Stuf Oreo---the Oreo $0$-Loaded---the stuf is pure
creme with no crumbs, so
\begin{equation}\label{eq:m-value}
    m_0 \;=\; \frac{m_s}{m_s + m_w}
       \;=\; \frac{10}{18}
       \;\approx\; 0.56.
\end{equation}
That is, a Mega Stuf Oreo is approximately 56\% filling and 44\%
wafer by mass.

\begin{remark}\label{rem:anderson-m}
Anderson's indirect calculation from~\S\ref{sec:product-line} gives
a different value: using the Mega Stuf multiplier~$2.68$ and the
original stuf fraction~$s = 0.29$, one obtains
$2.68\,s/(2.68\,s + 1 - s) \approx 0.52$.  The discrepancy
with~\eqref{eq:m-value} reflects measurement uncertainty across
different sources and methods.  We use the Lamers
value~\eqref{eq:m-value} throughout, as it involves fewer
intermediate steps.
\end{remark}

As $n$ increases, the crumbs in the filling add mass (see below),
so the stuf fraction~$m_n$ is not constant across the
sequence. We return to this
point after establishing the recursion.

\subsubsection*{The Stuf Loaded.}

The stuf of the Oreo Loaded consists of creme with small, visible wafer
crumbs distributed through it.\footnote{This required patience, a kitchen scale, and the assistance of a very interested dog. Lava's contribution to the data collection was essential.}

\begin{definition}[Loaded creme fraction]\label{def:loaded-creme-fraction}
Let $\ell$ denote the mass fraction of the stuf of the Oreo Loaded that
is creme. $1 - \ell$ corresponds to the part attributable to the wafer.
\end{definition}

Computing~$\ell$ requires two ingredients: the baseline creme
mass~$m_s$ from a plain Mega Stuf~\eqref{eq:mS-mW}, and the total
stuf mass of the Oreo Loaded.

For the Oreo Loaded, we extracted the filling from each of the
$21$~cookies in a single package and weighed the combined stuf.  The
total was $219$\,g, giving a mean stuf mass per cookie of
$219/21 = 10.4\overline{285714}$\,g, approximately $0.43$\,g heavier
than a plain Mega Stuf.  This excess is attributed to the wafer crumbs.

\begin{figure}[ht]
  \centering
  \framedimg[width=0.85\textwidth]{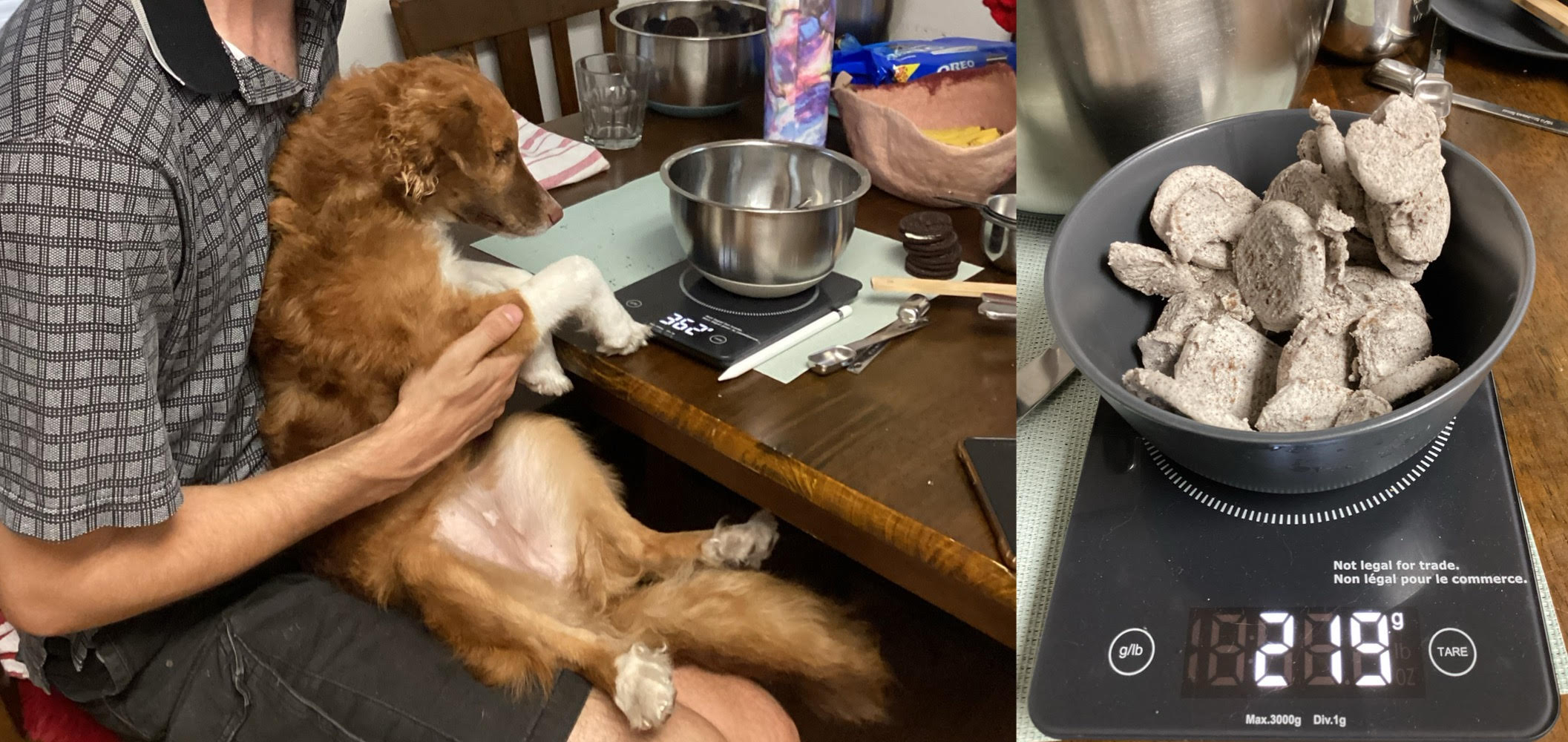}
  \caption{The measurement process.  Lava, providing quality
  control (left), and stuf from the Oreo Loaded on the kitchen
  scale (right).}\label{fig:methods}
\end{figure}

It remains to determine how the crumbs interact with the creme.  One
might expect the crumbs to displace an equal volume of creme, but
a density comparison rules this out: Oreo cookie crumbs have a bulk
density of approximately $0.46$\,g/cm$^3$~\cite{aquacalc}, while Oreo
creme has a density of approximately
$1.0$\,g/cm$^3$.\footnote{Estimated from Lamers' measurements of the
original Oreo stuf: $4$\,g, $4$\,mm thick, roughly $35$\,mm in
diameter~\cite{james_lamers}.}  Since the crumbs are less dense than
the creme they would replace, displacement would make the stuf
\emph{lighter}, not heavier.  The data show the opposite.

We instead assume that the creme, being a whipped filling with a porous
structure, accommodates the crumbs in its air pockets.  Under this model,
the crumbs add mass without displacing any creme, so the creme mass in
each Oreo Loaded stuf is the same $10$\,g as in a plain Mega Stuf.  The
loaded creme fraction is therefore
\begin{equation}\label{eq:ell-derivation}
    \ell \;=\; \frac{10}{219/21} \;=\; \frac{210}{219}
    \;\approx\; 0.96,
\end{equation}
That is, approximately 96\% of the filling mass is creme and 4\% is
wafer crumbs.  The crumbs are pure wafer: any creme from the crushed
Mega Stuf dissolves into the base creme upon mixing.

We regard this as an approximation;
a more precise determination would require the manufacturing
specifications from Mondelēz International.  

\subsubsection*{The recursion.}

\begin{wrapfigure}{R}{0.5\textwidth}
  \centering
  \vspace{-10pt}
  \framedimg[width=\linewidth]{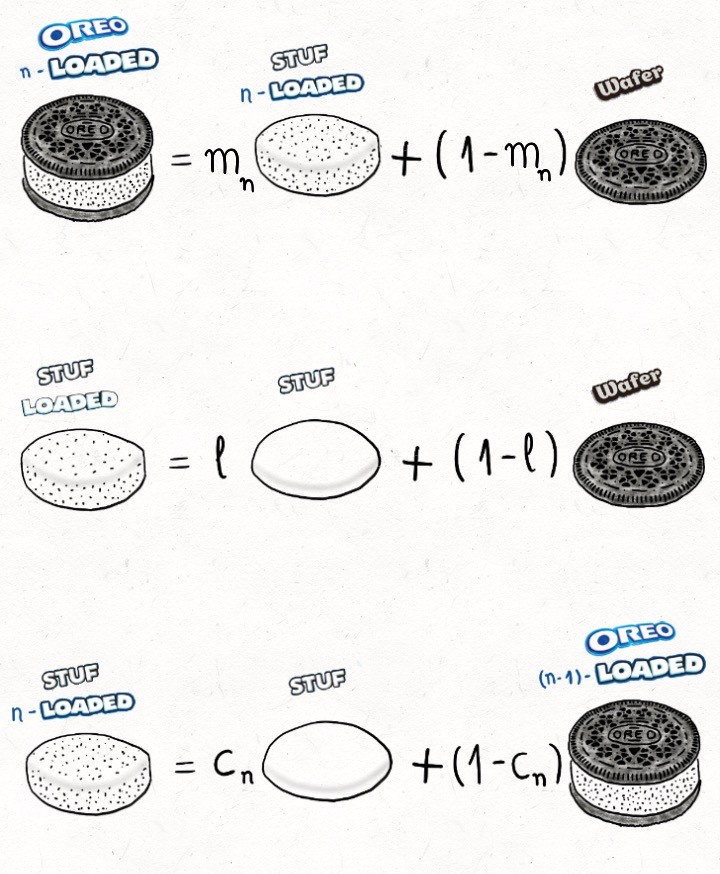}
  \caption{The three composition equations.}\label{fig:oreo-equations}
  \vspace{-8pt}
\end{wrapfigure}

We now have all the data we need.  It remains to write down how the creme
fraction of the stuf changes from one iteration to the next.

\begin{definition}[Creme fraction of the stuf]\label{def:creme-fraction}
Let $c_n$ denote the mass fraction of~$S_n$, the stuf of the Oreo
$n$-Loaded, that is creme.  The remaining fraction $1 - c_n$ is
attributable to wafer crumbs.
\end{definition}

Consider the Mega Stuf Oreo---the Oreo $0$-Loaded.  Its stuf is pure
creme with no crumbs, so
\begin{equation}\label{eq:c0}
    c_0 = 1.
\end{equation}

For $n \geq 1$, the stuf~$S_n$ is produced by combining two ingredients:
the base creme from a fresh Mega Stuf Oreo, and a whole Oreo
$(n{-}1)$-Loaded~$O_{n-1}$ that has been crushed and folded into the
filling.

\begin{definition}[Base creme fraction]\label{def:base-creme-fraction}
Let $p$ denote the mass fraction of~$S_n$ contributed by the base creme.
The remaining fraction $1 - p$ is contributed by the crushed Oreo
$(n{-}1)$-Loaded.
\end{definition}

Since the same physical process---the same base, the same amount of
filling, the same mixing proportions---is used at every step, $p$~does
not depend on~$n$.

The base creme is pure creme, contributing $p \cdot 1$ to the creme
fraction.  The crushed cookie~$O_{n-1}$ is a mixture of stuf and wafer
in proportions $m_{n-1}$ and $1 - m_{n-1}$; its stuf has creme
fraction~$c_{n-1}$, and its wafer contributes no creme.  The creme
fraction of~$S_n$ is therefore
\begin{equation}\label{eq:recursion}
    \boxed{\;c_n \;=\; p \;+\; (1 - p)\,m_{n-1}\,c_{n-1},
    \qquad n \geq 1.\;}
\end{equation}

This is the recursion that drives the construction.  The parameter~$p$
has not yet been determined, but it is not a free variable: it is pinned
down by the filling measurement of the Oreo Loaded.  At $n = 1$, the
recursion reads $c_1 = p + (1-p)\,m_0\,c_0$.  Since $c_0 = 1$ and
$c_1 = \ell$, this gives
\begin{equation}\label{eq:p-value}
    p \;=\; \frac{\ell - m_0}{1 - m_0}
      \;=\; \frac{210/219 - 10/18}{1 - 10/18}
      \;=\; \frac{265}{292}
      \;\approx\; 0.91.
\end{equation}

\subsubsection*{The stuf fraction recursion.}

The recursion~\eqref{eq:recursion} involves the stuf
fraction~$m_{n-1}$, which varies across iterations because the crumbs
in the filling add mass.  Write $w_n$ for the mass of wafer crumbs
in~$S_n$, so that
$m_n = (m_s + w_n)/(m_s + m_w + w_n)$, with $w_0 = 0$.

In~$S_n$, a fraction $1 - p$ of the stuf mass comes from the crushed
cookie~$O_{n-1}$, whose total wafer---biscuits and crumbs
together---has mass $m_w + w_{n-1}$ out of a total
$m_s + m_w + w_{n-1}$.  The wafer crumbs deposited in~$S_n$ therefore
satisfy
\[
    w_n \;=\; (1-p)\,(m_s + w_n)\,
    \frac{m_w + w_{n-1}}{m_s + m_w + w_{n-1}}.
\]
Solving for~$w_n$ gives the explicit recursion
\begin{equation}\label{eq:w-recursion}
    w_n \;=\;
    \frac{(1 - p)\,m_s\,(m_w + w_{n-1})}
         {m_s + p\,(m_w + w_{n-1})},
    \qquad w_0 = 0.
\end{equation}

Having established the recursions~\eqref{eq:recursion}
and~\eqref{eq:w-recursion} and pinned down all of their parameters, we
are ready to compute the $\infty$-Oreo.

\subsection{Solving the recursion}\label{sec:solving}

The recursion~\eqref{eq:recursion} involves the stuf
fraction~$m_{n-1}$, which itself depends on the crumb
mass~$w_n$ via~\eqref{eq:w-recursion}.  We solve for~$w^*$ first,
then use it to determine~$c^*$.

\begin{lemma}[Limiting stuf fraction]\label{lem:m-star}
The sequence~$(w_n)$ defined by~\eqref{eq:w-recursion} converges to
the unique positive root of
\begin{equation}\label{eq:w-quadratic}
    p\,(w^*)^2 \;+\; p\,(m_s + m_w)\,w^*
    \;-\; (1-p)\,m_s\,m_w \;=\; 0.
\end{equation}
The limiting stuf fraction is
$m^* = (m_s + w^*)/(m_s + m_w + w^*)$.
\end{lemma}

\begin{proof}
Write $f(w) = (1-p)\,m_s\,(m_w + w)\big/\bigl(m_s + p(m_w + w)\bigr)$
for the right-hand side of~\eqref{eq:w-recursion}.  Its derivative is
\[
    f'(w) \;=\;
    \frac{(1-p)\,m_s^{\,2}}{\bigl(m_s + p(m_w + w)\bigr)^2}
    \;\leq\; 1-p \;<\; 1
\]
for all $w \geq 0$, so $f$ is a contraction.  Since $f$ maps
$[0,\infty)$ into a bounded interval, the sequence converges to a
unique fixed point.  Setting $w^* = f(w^*)$ and clearing denominators
gives~\eqref{eq:w-quadratic}; the positive root is
\[
    w^* \;=\; \frac{-(m_s + m_w)
    + \sqrt{(m_s + m_w)^2 + \dfrac{4(1-p)}{p}\,m_s\,m_w\,}}{2}.
\]
\end{proof}

Substituting $m_s = 10$\,g, $m_w = 8$\,g, and
$p = 265/292$~\eqref{eq:p-value} gives
$w^* \approx 0.44$\,g and
\begin{equation}\label{eq:m-star}
    m^* \;\approx\; 0.566.
\end{equation}

With the limiting stuf fraction in hand, we pass to the limit
on both sides of~\eqref{eq:recursion}.  Setting
$c_n = c_{n-1} = c^*$ and $m_{n-1} = m^*$ gives
$c^* = p + (1-p)\,m^*\,c^*$, so
\begin{equation}\label{eq:fixed-point}
    c^* \;=\; \frac{p}{1 - (1-p)\,m^*}.
\end{equation}
The contraction factor at step~$n$ is $(1-p)\,m_{n-1}$, which is
bounded above by $1-p < 1$ uniformly in~$n$.  Since
$m_n \to m^*$ by Lemma~\ref{lem:m-star}, the
recursion~\eqref{eq:recursion} is a first-order affine recurrence with
convergent coefficients and contraction ratio bounded below~$1$.  The
sequence~$(c_n)$ therefore converges.

\begin{theorem}[The $\infty$-Oreo]\label{thm:infinity-oreo}
The creme fraction of the stuf of the $\infty$-Oreo is
\[
    c^* \;=\; \frac{p}{1 - (1-p)\,m^*}
       \;=\; \frac{265/292}{1 - (27/292) \times 0.566}
       \;\approx\; 0.958.
\]
That is, the stuf of the $\infty$-Oreo is approximately
$95.8\%$~creme and $4.2\%$~wafer.
\end{theorem}

\begin{proof}
The fixed-point equation is~\eqref{eq:fixed-point}.  Convergence
of~$(w_n)$ is given by Lemma~\ref{lem:m-star}; convergence
of~$(c_n)$ follows from the uniform contraction bound
$(1-p)\,m_{n-1} \leq 1-p < 1$.
\end{proof}

The stuf of the $\infty$-Oreo is close in composition to that of the
Oreo Loaded ($\ell \approx 0.96$), but the two are not identical.
Pinning down the exact difference requires more precise measurements
than those currently available.

\medskip

Note that the fixed point~\eqref{eq:fixed-point} does not depend
on~$c_0$.  This is not a coincidence.

\begin{corollary}[Independence of initial conditions]\label{cor:independence} 
The sequence~$(c_n)$ converges to~\eqref{eq:fixed-point} for every choice
of $c_0 \in [0,1]$.
\end{corollary}

\begin{proof}
The recursion~\eqref{eq:w-recursion} for~$w_n$ does not
involve~$c_0$, so $m^*$ and the fixed
point~\eqref{eq:fixed-point} are independent of the initial creme
fraction.  The contraction factor $(1-p)\,m_{n-1}$ is bounded above
by $1-p < 1$, so $|c_n - c^*| \to 0$ for every starting value.
\end{proof}

The $\infty$-Oreo is, in this sense, an attractor: no matter what cookie
one begins with, iterating the process of mixing it into a Mega Stuf base
always converges to the same composition.

\medskip

But the $\infty$-Oreo involves only one food: an Oreo references itself, and
the recursion stays within the world of Oreos.  What happens when two
\emph{different} foods reference each other?
\newpage


\section{Bi-$\infty$ Foods}\label{sec:bi-infinity}

\begin{wrapfigure}{R}{0.5\textwidth}
  \centering
  \vspace{-10pt}
  \framedimg[width=\linewidth]{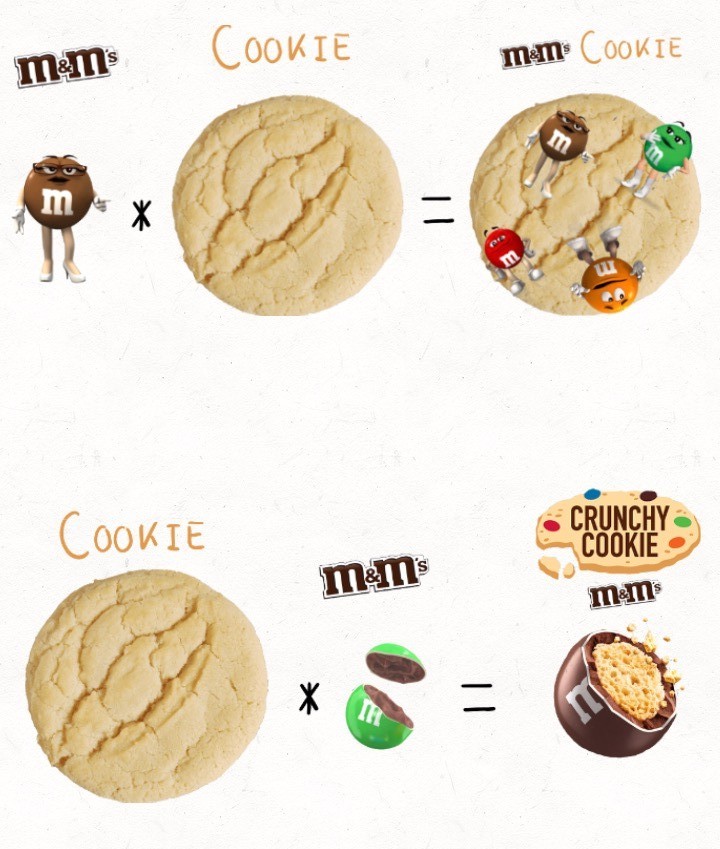}
  \caption{Food multiplication of M\&M and cookie.}\label{fig:mm-product}
  \vspace{-6pt}
\end{wrapfigure}

Recall that the Oreo Loaded is, at bottom, a Mega Stuf Oreo mixed with
a Mega Stuf Oreo.  We now abstract this operation to arbitrary pairs of
foods.

\begin{definition}[Food multiplication]\label{def:food-multiplication}
Given two foods $A$ and $B$, the \textbf{product} $A \ast B$ is the food
obtained by incorporating~$A$ into~$B$.  That is, $A \ast B$ is food~$B$
with food~$A$ mixed in.
\end{definition}

This notation recasts the construction of \S\ref{sec:infinity-oreo}. If we define powers recursively by $A^1\coloneqq A$ and $A^{n+1}\coloneqq A\ast A^n$, the Oreo Loaded becomes $\text{Mega Stuf}\;^2$:
\begin{align*}
    \text{Oreo Loaded}
    \;=\;& \text{Mega Stuf} \ast \text{Mega Stuf} \\
    \;=\;& \text{Mega Stuf}\;^2.
\end{align*}
    
More generally, the Oreo $n$-Loaded is the result of $n{+}1$~successive
multiplications:
\[
    \text{Oreo $n$-Loaded} \;=\; \text{Mega Stuf}\;^{n+1},
\]
and the $\infty$-Oreo is, in this notation, $\text{Mega Stuf}\;^\infty$.

The operation becomes more interesting when $A$ and $B$ are different
foods.  Consider M\&M's and cookies.

\begin{example}[M\&M Cookie]\label{ex:mm-cookie}
$\text{M\&M} \ast \text{Cookie}$ is a cookie with M\&M's mixed into the
dough: an \textbf{M\&M Cookie}.  This is a standard bakery item,
commercially available from a small number of producers that make them
without additional chocolate chips.\footnote{At the time of writing,
these include the Walmart in-store bakery and 4th Street Famous Cookies.
The absence of chocolate chips is essential: the product must be an
M\&M Cookie, not a chocolate chip cookie that also contains M\&M's.}
\end{example}

\begin{example}[Crunchy Cookie M\&M]\label{ex:crunchy-cookie-mm}
$\text{Cookie} \ast \text{M\&M}$ is an M\&M with cookie mixed into its
interior: a \textbf{Crunchy Cookie M\&M}.  This is also a commercially
available product, sold by Mars, Inc.\ under the M\&M's Crunchy Cookie
line.
\end{example}

Note that $\text{M\&M} \ast \text{Cookie}$ and $\text{Cookie} \ast
\text{M\&M}$ are not the same food. 

\begin{proposition}[Non-commutativity]\label{prop:non-commutative}
Food multiplication is not commutative.  In general,
$A \ast B \neq B \ast A$.
\end{proposition}

\begin{proof}
$\text{M\&M} \ast \text{Cookie}$ is a cookie.
$\text{Cookie} \ast \text{M\&M}$ is a candy.
\end{proof}

The existence of both products is what makes bi-$\infty$ foods
possible.  We now iterate.

\subsection{Coupled recursion}\label{sec:coupled-recursion}

With two different foods, the iteration splits into two sequences that
feed into each other.  Write $m_{n}$ for M\&M and $C$ for Cookie.  The two
products from the previous subsection are
\[
    M \ast C \;=\; \text{M\&M Cookie},
    \qquad
    C \ast M \;=\; \text{Crunchy Cookie M\&M}.
\]
These are the starting points---the analogues of the Oreo Loaded.

\subsubsection*{First iteration.}

\begin{wrapfigure}{R}{0.5\textwidth}
  \centering
  \vspace{-10pt}
  \framedimg[width=\linewidth]{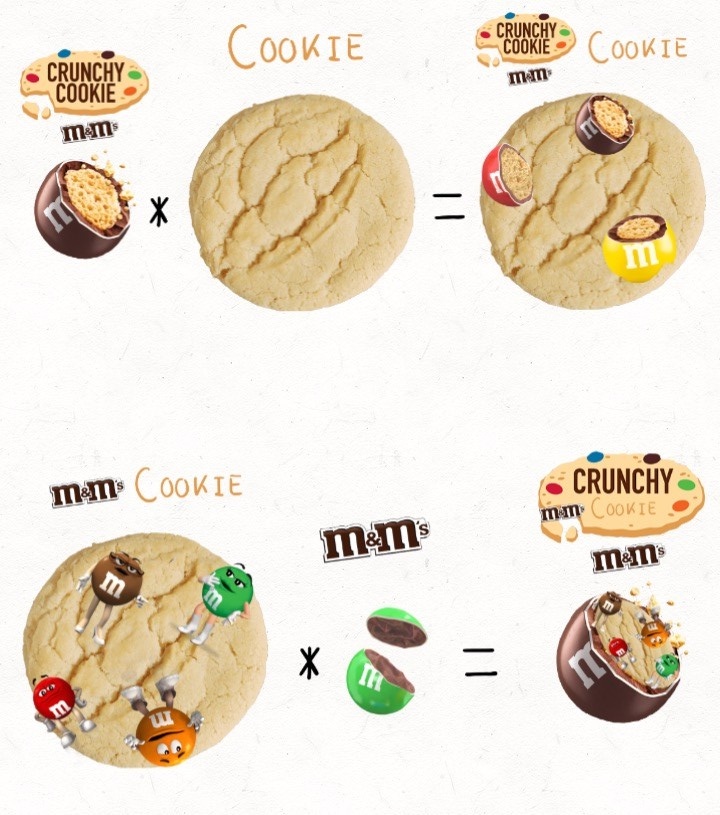}
  \caption{The first iteration of the coupled recursion.}\label{fig:mm-iteration}
  \vspace{-8pt}
\end{wrapfigure}

On the cookie side, instead of mixing plain M\&M's into a cookie, we
mix in Crunchy Cookie M\&M's---M\&M's that already contain cookie.
The result is
\[
    (C \ast M) \ast C,
\]
a cookie whose M\&M's have cookie inside them: the \textbf{Crunchy
Cookie M\&M Cookie}.

On the candy side, instead of mixing plain cookie into an M\&M, we mix
in an M\&M Cookie---a cookie that already contains M\&M's.  The result
is
\[
    (M \ast C) \ast M,
\]
an M\&M whose cookie interior contains M\&M's: the \textbf{Crunchy
M\&M Cookie M\&M}.

The cookie at this step uses the candy from the
previous step, and vice versa.

\subsubsection*{The general case.}

At each step, each sequence absorbs the latest
product from the other.

\begin{definition}[Coupled recursion]\label{def:coupled-recursion}
Define two sequences of foods $(P_n)_{n \geq 0}$ and $(Q_n)_{n \geq 0}$
by
\[
    P_0 \;\coloneqq\; M \ast C, \qquad Q_0 \;\coloneqq\; C \ast M,
\]
and, for $n \geq 0$,
\begin{equation}\label{eq:coupled-recursion}
    P_{n+1} \;\coloneqq\; Q_n \ast C, \qquad Q_{n+1} \;\coloneqq\; P_n \ast M.
\end{equation}
\end{definition}

Here $P_n$ is always a cookie and $Q_n$ is always an M\&M, but the
interior of each grows one layer deeper at every step: the cookie~$P_n$
contains an M\&M that contains a cookie that contains an M\&M,
$n$~layers deep.

This is a coupled recurrence relation that has a limit.

\begin{definition}[Bi-$\infty$ M\&M Cookie]\label{def:bi-infinity-mm}
The \textbf{$\infty$-M\&M Cookie} and the \textbf{$\infty$-Crunchy
Cookie M\&M} are the limits of the coupled sequences:
\[
    P_\infty \;\coloneqq\; \lim_{n \to \infty} P_n,
    \qquad
    Q_\infty \;\coloneqq\; \lim_{n \to \infty} Q_n.
\]
Together, they form a \textbf{bi-$\infty$ food}: a pair of foods, each
of which contains the other, all the way down.
\end{definition}

By Proposition~\ref{prop:non-commutative}, these two limits are not the
same food: one is a cookie and the other is a candy, no matter how many
layers of recursion one performs.  No physical cookie can contain
infinitely many alternating layers, but the mass fractions converge as
$n \to \infty$, and the limiting ratios are concrete.  One could produce
a deconstructed version of either limit food by mixing the ingredients in
the computed proportions.

\begin{figure}[ht]
  \centering
  \framedimg[width=0.80\textwidth]{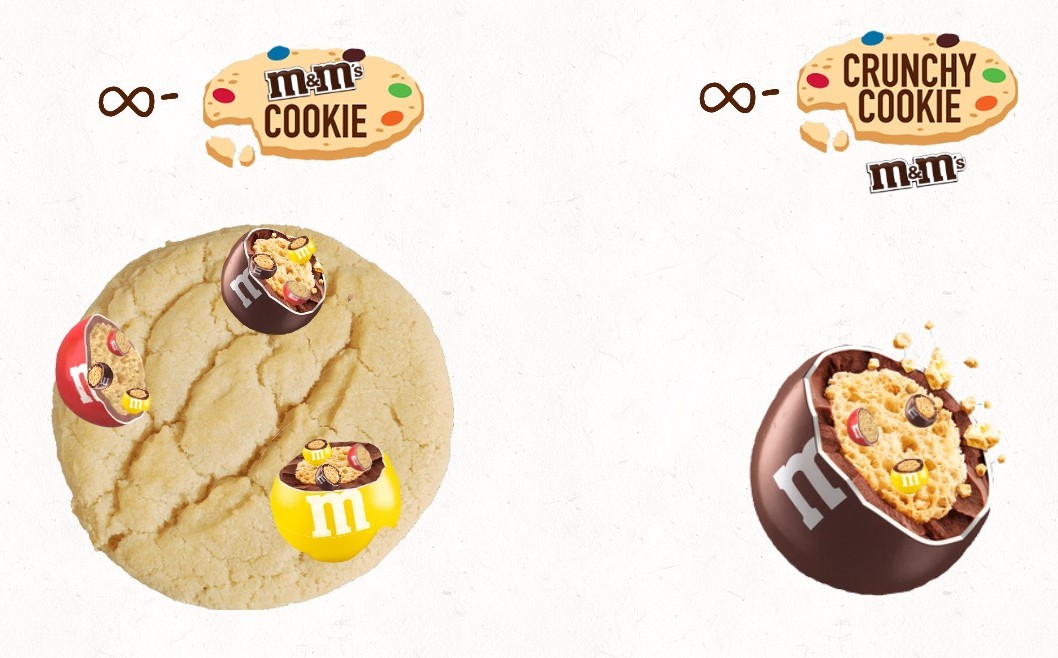}
  \caption{The  bi-$\infty$ food pair $\infty$-M\&M Cookie
  (left) and $\infty$-Crunchy Cookie
  M\&M (right) is the limit of the coupled recursion.   }\label{fig:bi-infinity}
\end{figure}

To determine these proportions, we need to set up and solve the system of
equations governing the coupled recursion.

\subsection{The composition equations}\label{sec:bi-composition}

In \S\ref{sec:composition-equations}, we tracked a single quantity---the
creme fraction of the stuf---across iterations.  The coupled recursion
requires two, one for each sequence.

Every product in both sequences is, at the level of raw ingredients, a
mixture of two fundamental materials: \textbf{cookie dough} and
\textbf{M\&M material} (candy shell and chocolate).  Our goal is to
determine how the proportions of these two materials evolve, and what
they converge to.

This requires two measured quantities and a pair of equations.  We
introduce the measurements first.

\subsubsection*{The M\&M Cookie composition.}

The M\&M Cookie~$P_n$ is always a cookie: a base of cookie dough with
some amount of M\&M product mixed in.  Since the same recipe is used at
every step (the same dough, the same quantity of mix-in), the overall
split between dough and mix-in is the same for every~$n$.

\begin{definition}[M\&M mix-in fraction]\label{def:mu}
Let $\mu$ denote the mass fraction of an M\&M Cookie that consists of
M\&M mix-ins.  The remaining fraction $1 - \mu$ is cookie dough.
\end{definition}

\subsubsection*{The Crunchy Cookie M\&M composition.}

\begin{wrapfigure}{R}{0.5\textwidth}
  \centering
  \vspace{-10pt}
  \framedimg[width=\linewidth]{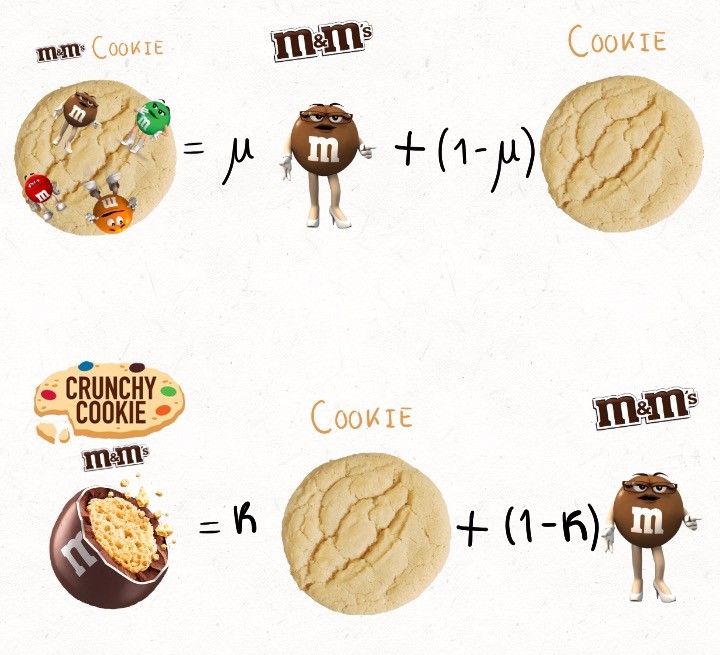}
  \caption{The composition equations for the M\&M Cookie system.}\label{fig:mm-equations}
  \vspace{-8pt}
\end{wrapfigure}

The Crunchy Cookie M\&M~$Q_n$ is, at every step, an M\&M: a base of
M\&M material with some amount of cookie product mixed in.

\begin{definition}[Cookie mix-in fraction]\label{def:kappa}
Let $\kappa$ denote the mass fraction of a Crunchy Cookie M\&M that
consists of cookie mix-ins.  The remaining fraction $1 - \kappa$ is M\&M
material.
\end{definition}

The parameters~$\mu$ and~$\kappa$ play the same role here that~$\ell$ played in the Oreo analysis: they are empirical inputs,
determined by the recipes, that pin down the recursion.

\subsubsection*{The recursion.}

We now track how the composition of each product changes across
iterations.

\begin{definition}[M\&M fraction and cookie fraction]\label{def:an-bn}
Let $a_n$ denote the mass fraction of~$P_n$ that is M\&M material, and
let $b_n$ denote the mass fraction of~$Q_n$ that is cookie dough.
\end{definition}

At step~$0$, the M\&M's inside the cookie are plain M\&M's (pure M\&M
material), and the cookie inside the M\&M is plain cookie (pure dough),
so
\begin{equation}\label{eq:ab-initial}
    a_0 = \mu, \qquad b_0 = \kappa.
\end{equation}

For $n \geq 0$, the cookie~$P_{n+1}$ is produced by mixing~$Q_n$ into a
fresh cookie base.  A fraction~$1 - \mu$ of~$P_{n+1}$ is cookie dough,
contributing no M\&M material.  The remaining fraction~$\mu$ is~$Q_n$,
of which a fraction $1 - b_n$ is M\&M material and $b_n$ is cookie
dough.  The M\&M material in~$P_{n+1}$ therefore comes entirely from
the~$Q_n$ portion:
\[
    a_{n+1} \;=\; \mu\,(1 - b_n).
\]
Symmetrically, the M\&M~$Q_{n+1}$ is produced by mixing~$P_n$ into a
fresh M\&M base.  A fraction~$1 - \kappa$ is M\&M material and the
remaining fraction~$\kappa$ is~$P_n$, of which $1 - a_n$ is cookie
dough.  So:
\[
    b_{n+1} \;=\; \kappa\,(1 - a_n).
\]

Together, these give the coupled recursion:
\begin{equation}\label{eq:bi-recursion}
    \boxed{\;a_{n+1} \;=\; \mu\,(1 - b_n),
    \qquad
    b_{n+1} \;=\; \kappa\,(1 - a_n),
    \qquad n \geq 0.\;}
\end{equation}

Note the coupling: the M\&M content of the cookie at step~$n{+}1$
depends on the cookie content of the M\&M at step~$n$, and vice versa.

\subsection{Solving the coupled recursion}\label{sec:bi-solving}

The system~\eqref{eq:bi-recursion} is a pair of affine recurrences, each
feeding into the other.  It can be solved by the same fixed-point
strategy used in \S\ref{sec:solving}.

If the sequences converge to limits~$a^*$ and~$b^*$, then passing to the
limit on both sides of~\eqref{eq:bi-recursion} gives
\[
    a^* \;=\; \mu\,(1 - b^*),
    \qquad
    b^* \;=\; \kappa\,(1 - a^*).
\]
Substituting the second equation into the first:
\[
    a^* \;=\; \mu\bigl(1 - \kappa(1 - a^*)\bigr)
         \;=\; \mu\,(1 - \kappa) \;+\; \mu\kappa\,a^*.
\]
Solving for~$a^*$, and repeating for~$b^*$:
\begin{equation}\label{eq:bi-fixed-point}
    a^* \;=\; \frac{\mu\,(1-\kappa)}{1 - \mu\kappa},
    \qquad
    b^* \;=\; \frac{\kappa\,(1-\mu)}{1 - \mu\kappa}.
\end{equation}

It remains to verify convergence.  Substituting the recursion into
itself, we find that each sequence satisfies a decoupled recurrence at
every other step:
\begin{equation}\label{eq:bi-decoupled}
    a_{n+2} \;=\; \mu(1-\kappa) \;+\; \mu\kappa\,a_n,
    \qquad
    b_{n+2} \;=\; \kappa(1-\mu) \;+\; \mu\kappa\,b_n.
\end{equation}
Subtracting~\eqref{eq:bi-fixed-point}
from~\eqref{eq:bi-decoupled}:
\[
    a_{n+2} - a^* \;=\; \mu\kappa\,(a_n - a^*),
    \qquad
    b_{n+2} - b^* \;=\; \mu\kappa\,(b_n - b^*).
\]
Since $0 < \mu < 1$ and $0 < \kappa < 1$, the contraction factor
$\mu\kappa$ lies in $(0,1)$, and both sequences converge geometrically.

\begin{theorem}[Bi-$\infty$ M\&M Cookie]\label{thm:bi-infinity}
The $\infty$-M\&M Cookie has M\&M-material fraction
\[
    a^* \;=\; \frac{\mu\,(1-\kappa)}{1 - \mu\kappa}
\]
and cookie-dough fraction $1 - a^*$.  The $\infty$-Crunchy Cookie M\&M
has cookie-dough fraction
\[
    b^* \;=\; \frac{\kappa\,(1-\mu)}{1 - \mu\kappa}
\]
and M\&M-material fraction $1 - b^*$.
\end{theorem}

\begin{proof}
The fixed-point equations and the convergence argument are given above;
see~\eqref{eq:bi-fixed-point} and~\eqref{eq:bi-decoupled}.
\end{proof}

The two products are always compositionally distinct: the cookie
always has more cookie dough than the M\&M does, and vice versa.
To see this, note that equal cookie-dough fractions would require
$a^* + b^* = 1$, but
\[
    a^* + b^* \;=\; \frac{\mu + \kappa - 2\mu\kappa}{1 - \mu\kappa}
    \;<\; 1
\]
for all $\mu, \kappa \in (0,1)$.  The non-commutativity
of Proposition~\ref{prop:non-commutative} persists in the limit.

The iteration makes each product \emph{more like itself}.
The M\&M-material fraction of the cookie decreases from $a_0 = \mu$ to
$a^* = \mu(1-\kappa)/(1-\mu\kappa) < \mu$, and similarly $b^* < \kappa$.
This is because the M\&M's being mixed into the cookie are no longer
pure M\&M---they contain cookie, which dilutes the foreign component.
Each product converges toward a purer version of itself, though neither
reaches purity as long as $\mu, \kappa > 0$.  As in the Oreo case,
the contraction factor~$\mu\kappa$ drives any starting point to the
same fixed point, so the limits are independent of initial conditions.

\begin{remark}\label{rem:bi-measurements}
The formulas in Theorem~\ref{thm:bi-infinity} are exact, but their
numerical evaluation requires measurements of~$\mu$ and~$\kappa$.
These measurements have not yet been carried out.  A
determined reader with a kitchen scale, a bag of Crunchy Cookie M\&M's,
and a box of M\&M Cookies from their preferred manufacturer is invited to
complete the computation.
\end{remark}


\section{Classification of $\infty$-Foods}\label{sec:classification}

The $\infty$-Oreo is built from a single food iterating with itself;
the $\infty$-M\&M Cookie from two foods iterating with each other.
In this section we classify $\infty$-foods by the number of foods
involved in the recursion and begin to study the structures that
govern them.

\subsection{Mono-$\infty$ foods}\label{sec:mono-infinity}

We begin with the case of a single food.

\begin{definition}[Mono-$\infty$ food]\label{def:mono-infinity}
A food~$A$ forms a \textbf{mono-$\infty$ food} if the product $A^2 = A \ast A$
exists as a commercially available food.  The \textbf{$\infty$-$A$} is
the limit $A^\infty \coloneqq \lim_{n \to \infty} A^n$.
\end{definition}

The obstacle to producing a mono-$\infty$ food is not mathematical but
commercial.  Once~$A^2$ exists, the recurrence is determined, and the
limit follows by the contraction argument of
\S\ref{sec:solving}.  The hard part is that most foods are flavored
with something other than themselves; a product whose interior contains
a version of itself is an unusual thing to manufacture.

The Oreo Loaded is the only example known to the author: it
is, as established in \S\ref{sec:product-line}, a Mega Stuf Oreo mixed
with a Mega Stuf Oreo, and hence the $\infty$-Oreo exists by
Theorem~\ref{thm:infinity-oreo}.  There are rumors of a
KitKat-flavored KitKat produced in Japan, where the KitKat product line
is known for its extensive roster of regional
flavors.\footnote{Japanese KitKat flavors include, among many others,
matcha, sake, strawberry cheesecake, and sweet
potato~\cite{kitkat_japan}.  A self-referential entry in this list would
not be entirely out of character.}  No photographic evidence, product
listing, or official confirmation has been located.  Until then, the $\infty$-Oreo remains the only confirmed mono-$\infty$ food.

\subsection{Bi-$\infty$ foods}\label{sec:bi-infinity-foods}

We now allow two foods.

\begin{definition}[Bi-$\infty$ food]\label{def:bi-infinity}
Two foods $A$ and $B$ form a \textbf{bi-$\infty$ food} pair if both products
$A \ast B$ and $B \ast A$ exist as commercially available foods.  The
\textbf{$\infty$-$(A,B)$ pair} consists of the two limits
\[
    (A \ast B)^\infty \;\coloneqq\; \lim_{n \to \infty} P_n,
    \qquad
    (B \ast A)^\infty \;\coloneqq\; \lim_{n \to \infty} Q_n,
\]
where $(P_n)$ and $(Q_n)$ are the coupled sequences of
Definition~\ref{def:coupled-recursion}.
\end{definition}

The requirement that \emph{both} products exist is essential.  A cookie
with M\&M's inside it is not, by itself, enough to produce a bi-$\infty$
food; one also needs an M\&M with cookie inside it.  The two products
can be manufactured by different companies, and neither has any reason to
exist just because the other does.  That both happen to be commercially
available is rare, but not as rare as the self-referential products
required for mono-$\infty$ foods.

The M\&M Cookie and the Crunchy Cookie M\&M form the primary example,
analyzed in \S\ref{sec:bi-infinity}.  Other candidates include Taco
Bell's Doritos Locos Tacos and Frito-Lay's taco-flavored Doritos
($A = \text{Doritos}$, $B = \text{Taco}$), as well as ramen-flavored
Pringles and Pringles-flavored ramen, available in Japan too
($A = \text{Pringles}$, $B = \text{Ramen}$).  In each case, the coupled
recursion of \S\ref{sec:coupled-recursion} applies and the $\infty$-pair
is computable in principle.  The numerical evaluation is left as an open
problem, pending the development of a reliable method for measuring how
much taco is in a Dorito.

\subsection{Tri-$\infty$ foods}\label{sec:tri-infinity}

Mono-$\infty$ foods are built from a single food; bi-$\infty$ foods from
a pair.  We now ask whether three foods can be arranged so that each one
references the next, with the last referencing the first.

Rather than starting with a definition, we start with an attempt to build
one.

\medskip

Take an Oreo and crush it into  ice cream.  The result is
cookies-and-cream ice cream: ice cream with Oreo inside.  This is
available at almost every supermarket, and it gives us our first
arrow:
\[
    \text{Oreo} \;\longrightarrow\; \text{Ice cream}.
\]

Now take that ice cream and put it inside a cake.  The result is an ice
cream cake---a common phenomenon in United States birthday
celebrations.\footnote{The author, raised in the Canary Islands, first
encountered the ice cream cake upon arriving in the United States and has
not yet fully come to terms with it.}  This gives a second arrow:
\[
    \text{Ice cream} \;\longrightarrow\; \text{Cake}.
\]

To close the loop, we need to put a cake inside an Oreo.  This seems
difficult---until one remembers that we are probably at a birthday
party.  Recall the \birthdaycake{} from \S\ref{sec:motivations}, 
an Oreo whose filling is flavored with birthday cake, regardless of what flavor of cake one is celebrating with.  We now have the
third arrow:
\[
    \text{Cake} \;\longrightarrow\; \text{Oreo}.
\]

The three arrows form a closed loop.  Every arrow corresponds to a
commercially available product.

\begin{figure}[ht]
  \centering
  \framedimg[width=0.82\textwidth]{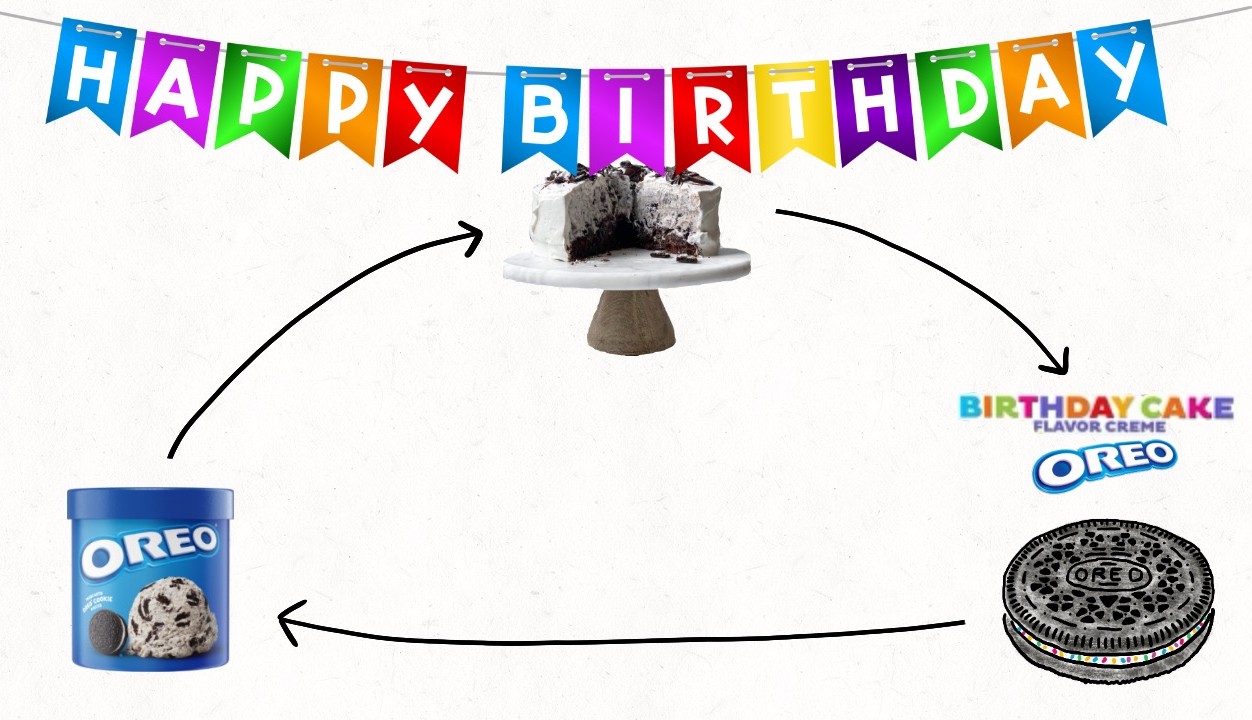}
  \caption{A tri-$\infty$ food.  Oreos go into ice cream, ice cream
  goes into cake, and cake---via a birthday---goes back into an Oreo.
  Every arrow corresponds to a product available at a
  supermarket.}\label{fig:tri-infinity}
\end{figure}

\begin{definition}[Tri-$\infty$ food]\label{def:tri-infinity}
A \textbf{tri-$\infty$ food} is a system of three foods $A$, $B$, $C$
forming a cycle: the products $A \ast B$, $B \ast C$, and $C \ast A$ all
exist as commercially available foods.
\end{definition}

\begin{example}[The Oreo--ice cream--cake cycle]\label{ex:tri-infinity}
The cycle constructed above,
\[
    \text{Oreo}
    \;\xrightarrow{\;\text{cookies \& cream}\;}\;
    \text{Ice cream}
    \;\xrightarrow{\;\text{ice cream cake}\;}\;
    \text{Cake}
    \;\xrightarrow{\;\text{\birthdaycake}\;}\;
    \text{Oreo},
\]
is a tri-$\infty$ food.  The iteration proceeds by traversing the cycle
repeatedly: at each pass, the food being incorporated is itself the
product of the previous traversal.  After $n$~full cycles, each food
contains $n$~nested layers of the other two.  In the limit, one obtains
three $\infty$-foods---an $\infty$-cookies-and-cream ice cream, an
$\infty$-ice cream cake, and an $\infty$-\birthdaycake{}---each
containing the other two.
\end{example}

The construction of this example reveals a pattern.  In every case we
have studied, the condition for producing an $\infty$-food is the same:
one needs a \emph{cycle}.  A mono-$\infty$ food is a self-loop.  A
bi-$\infty$ food is a cycle of length two.  A tri-$\infty$ food is a
cycle of length three.  The relevant structure is not the number of
foods, but the topology of the arrows connecting them.  This observation
leads to a more general theory.

\subsection{Homological foods}\label{sec:homological}

We now give this observation a formal home.

\begin{definition}[Food quiver]\label{def:food-quiver}
A \textbf{food quiver} is a directed graph whose
vertices are foods and whose arrows represent mixing: there is an arrow
$A \to B$ whenever the product $A \ast B$---the food obtained by
incorporating~$A$ into~$B$---exists as a commercially available
product.
\end{definition}

To see an example of what this graph looks like, it helps to start from a food that
participates in many arrows at once.

Consider a pint of Ben \& Jerry's Milk \& Cookies ice cream.  Among its
ingredients are cookie pieces, giving an arrow from Cookie to Ice cream.  But
ice cream can also go inside a cookie: an ice cream sandwich is precisely
this.  So we have a reverse arrow as well, and therefore a two-cycle
between Cookie and Ice cream.

The same pint also contains what appear to be Oreo pieces:
an arrow from Oreo to Ice cream, and an Oreo ice cream sandwich gives
the reverse.  The Oreo--ice cream--cake cycle of
Example~\ref{ex:tri-infinity} contributes two more arrows, and the Oreo
vertex carries its self-loop from \S\ref{sec:infinity-oreo}.

We are not done.  M\&M's go inside cookies, and cookies go inside M\&M's.  This is
the bi-$\infty$ system of \S\ref{sec:bi-infinity}, and it attaches
two more arrows to the Cookie vertex, forming another two-cycle.

\begin{figure}[ht]
  \centering
  \framedimg[width=0.88\textwidth]{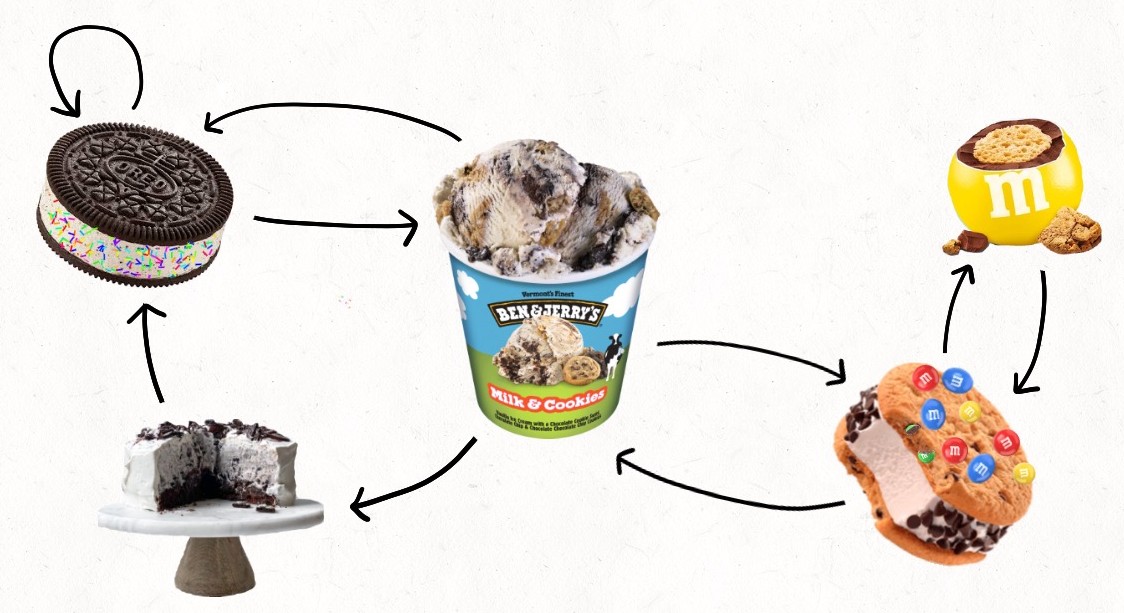}
  \caption{Example of a food quiver.  Five foods, connected by
commercially available products.  Ice cream sits
at the center, linking every cycle in the graph via Ben \& Jerry's Milk \& Cookies. Conjecture ~\ref{conj:quiver-topology} asks how compositions would change if we put M\&Ms inside that ice cream.}\label{fig:food-quiver}
\end{figure}

Five foods, and already the quiver contains multiple overlapping cycles.
The Ice cream vertex sits at the center: a single pint of Ben \& Jerry's
Milk \& Cookies witnesses arrows from both Cookie and Oreo, linking the
Oreo triangle on the left of Figure~\ref{fig:food-quiver} to the
M\&M--Cookie pair on the right.

\begin{definition}[Homological food]\label{def:homological-food}
A \textbf{homological food} is an $\infty$-food whose recursive
structure is determined by a food quiver.
\end{definition}

Every $\infty$-food constructed in this paper is a homological food.
The term is chosen by analogy with algebraic topology, where the cycles
in a space determine its homology.  In the food quiver, cycles determine
which $\infty$-foods exist; the question is whether the topology of the
quiver determines anything further.

\begin{conjecture}\label{conj:quiver-topology}
The topology of the food quiver determines the convergence properties and
limit compositions of its $\infty$-foods.  In particular, the presence
of additional cycles passing through a vertex alters the limit
compositions in a predictable way.
\end{conjecture}

For instance, the $\infty$-M\&M Cookie of
Theorem~\ref{thm:bi-infinity} was computed using only the two-cycle
between M\&M and Cookie.  But the quiver in
Figure~\ref{fig:food-quiver} shows that Cookie also connects to Ice
cream.  If M\&M's were added to the ice cream---creating a new
three-cycle through M\&M, Ice cream, and Cookie---the recursion at the
Cookie vertex would flow through a longer path, and the limit composition
of the $\infty$-M\&M Cookie could change.

We believe that the tools of algebraic topology and representation theory
may have something to say about these homological foods.  This is left as
an invitation.


\section{Discussion}\label{sec:discussion}

The $\infty$-Oreo is a fixed point of an affine recurrence whose
parameters are determined by the composition of commercially available
cookies.  To produce one, there is no need to iterate: one simply mixes
creme and wafer crumbs in the ratio prescribed by
Theorem~\ref{thm:infinity-oreo}, gives the resulting filling the volume
of a Mega Stuf, and presses it between two wafers.  Every ingredient is
available at a grocery store.  The mathematics has already performed the
infinite iteration; the baker only needs to read off the answer. This is a theory of commercial infinity.

But it is a theory, and theories rest on measurements.  The numerical
inputs to this paper ($s = 0.29$ for the original Oreo, $m_0 \approx
0.56$ for the Mega Stuf, $\ell \approx 0.96$ for the Oreo Loaded)
were drawn from magazine articles, a high school classroom experiment,
and a kitchen scale shared with a dog.  The qualitative conclusions are
robust: the recursion converges and the limit is independent of initial
conditions.  The third
decimal place of~$c^*$ is not to be trusted.  A more precise
determination would require either a careful independent measurement
campaign or, ideally, the manufacturing specifications from Mondelēz
International.

It is worth noting what the $\infty$-Oreo is \emph{not}.  It is not an
infinitely large cookie, nor a cookie containing infinitely many layers
of nested Oreo.  It is a cookie of ordinary size whose filling happens
to be mixed in the limiting proportion.  The infinity is in the
mathematics, not in the cookie.

\medskip

The most immediate open problem is empirical: measure~$\mu$
and~$\kappa$ to evaluate the bi-$\infty$ M\&M Cookie formulas of
Theorem~\ref{thm:bi-infinity}.  The equipment required is a kitchen
scale and a free afternoon.

On the theoretical side, Conjecture~\ref{conj:quiver-topology} asks
whether the topology of the food quiver affects limit compositions.  A
natural first case is to compare the $\infty$-M\&M Cookie computed from
the two-cycle alone with the limit obtained when additional cycles pass
through the Cookie vertex.  Whether extra cycles alter the fixed point,
accelerate convergence, or leave the limit unchanged is unknown.

More ambitiously: is there a spectral theory of food quivers?  Each
arrow carries a mixing fraction and each vertex a base composition, so
the quiver is a weighted directed graph with algebraic data on its
cells, reminiscent of a cellular
sheaf~\cite{hansen_ghrist_spectral}.  Whether the tools of spectral
sheaf theory have anything to say about recursive snacking is an open
and, we believe, entertaining question.

Finally, the entire framework suggests a design problem.  Given a target
composition for a filling, can a food company engineer the mixing
fractions of its products so that the associated recursion converges to
that target?  The fixed-point formula~\eqref{eq:fixed-point} is
invertible: given a desired~$c^*$, one can solve for the mixing
parameters that produce it.  In that sense, the $\infty$-Oreo is not just a
mathematical curiosity.  It is a recipe.


\section*{Acknowledgments}

This paper grew out of a conversation during tea time with Anne Somalwar, Ana Pavlakovi\'c, Josh Tabb, Matthew Stevens, and Collin Free, and was enhanced by further discussions with Victor G. Bennett, all of whom are welcome to
share in the credit and the blame.  The author is grateful to the
University of Pennsylvania Mathematics Department  \href{https://pizza-seminar.github.io/}{Pizza Seminar}  and its organizers, Riley Shahar and Chris Dunstan, for
providing the venue for the original talk, and to its audience for
taking the subject exactly as seriously as it deserved.

The author thanks James Lamers, Dan Anderson and his students at Queensbury High
School for the measurements that made the Mega Stuf composition
computable, and Lava and Marta for their unwavering interest in the data collection
process.


\bibliographystyle{abbrvnat}
\bibliography{references}

\end{document}